\documentclass{amsart}

\usepackage{amsmath,
            amssymb,
            bibentry,
            enumitem,
            graphicx,
            mathrsfs,
            tabu,
            xypic}
\usepackage[comma,numbers,sort,square]{natbib}
\usepackage[colorlinks=true,
            linktocpage=true,
            linkcolor=magenta,
            citecolor=magenta,
            urlcolor=magenta]{hyperref}
\PassOptionsToPackage{hyphens}{url}

\newtheorem{theorem}{Theorem}[section]
\newtheorem{proposition}[theorem]{Proposition}
\newtheorem{lemma}[theorem]{Lemma}
\newtheorem{corollary}[theorem]{Corollary}
\theoremstyle{definition}
\newtheorem{definition}[theorem]{Definition}

\newtheoremstyle{dotless}{}{}{\upshape}{}{\bfseries}{}{ }{}
\theoremstyle{dotless}

\newtheorem{innerhighlight}{}
\newenvironment{highlight}[1]
  {\renewcommand\theinnerhighlight{#1}\innerhighlight}
  {\endinnerhighlight}

\DeclareMathOperator{\res}{\upharpoonright}
\newcommand{\ran}{\operatorname{ran}}

\newcommand{\height}{\operatorname{ht}}

\newcommand{\seq}[1]{\langle #1 \rangle}

\newcommand{\RCA}{\mathsf{RCA}}

\newcommand{\RT}{\mathsf{RT}}
\newcommand{\COH}{\mathsf{COH}}
\newcommand{\CAC}{\mathsf{CAC}}
\newcommand{\ADS}{\mathsf{ADS}}
\newcommand{\SRT}{\mathsf{SRT}}
\newcommand{\SCAC}{\mathsf{SCAC}}
\newcommand{\SADS}{\mathsf{SADS}}

\newcommand{\D}{\mathsf{D}}
\newcommand{\ADC}{\mathsf{ADC}}
\newcommand{\SADC}{\mathsf{SADC}}
\newcommand{\GeneralSADS}{\mathsf{General}\text{-}\mathsf{SADS}}
\newcommand{\GeneralSADC}{\mathsf{General}\text{-}\mathsf{SADC}}
\newcommand{\WSCAC}{\mathsf{WSCAC}}

\newcommand{\cequiv}{\equiv_{\text{\upshape c}}}
\newcommand{\cred}{\leq_{\text{\upshape c}}}
\newcommand{\ncred}{\nleq_{\text{\upshape c}}}
\newcommand{\scred}{\leq_{\text{\upshape sc}}}

\newcommand{\ured}{\leq_{\text{\upshape W}}}
\newcommand{\nured}{\nleq_{\text{\upshape W}}}
\newcommand{\sured}{\leq_{\text{\upshape sW}}}

\renewcommand{\small}{\text{\upshape S}}
\renewcommand{\large}{\text{\upshape L}}
\newcommand{\isolated}{\text{\upshape I}}
\newcommand{\FinLin}{\text{\upshape FinLO}}
\newcommand{\FinPO}{\text{\upshape FinPO}}
 
\begin{document}

\title{The uniform content of partial and linear orders}

\author[E.\ P.\ Astor]{Eric P. Astor}
\address{Department of Mathematics\\
University of Connecticut\\
Storrs, Connecticut U.S.A.}
\email{eric.astor@uconn.edu}

\author[D.\ D.\ Dzhafarov]{Damir D. Dzhafarov}
\address{Department of Mathematics\\
University of Connecticut\\
Storrs, Connecticut U.S.A.}
\email{damir@math.uconn.edu}

\author[R.\ Solomon]{Reed Solomon}
\address{Department of Mathematics\\
University of Connecticut\\
Storrs, Connecticut U.S.A.}
\email{david.solomon@uconn.edu}

\author[J.\ Suggs]{Jacob Suggs}
\address{Department of Mathematics\\
University of Connecticut---Waterbury\\
Waterbury, Connecticut U.S.A.}
\email{jacob.suggs@uconn.edu}

\thanks{Dzhafarov was partially supported by NSF grant DMS-1400267.}

\maketitle

\begin{abstract}
	The principle $\ADS$ asserts that every linear order on $\omega$ has an infinite ascending or descending sequence. This has been studied extensively in the reverse mathematics literature, beginning with the work of Hirschfeldt and Shore \cite{HS-2007}. We introduce the principle $\ADC$, which asserts that every such linear order has an infinite ascending or descending chain. The two are easily seen to be equivalent over the base system $\RCA_0$ of second order arithmetic; they are even computably equivalent. However, we prove that $\ADC$ is strictly weaker than $\ADS$ under Weihrauch (uniform) reducibility. In fact, we show that even the principle $\SADS$, which is the restriction of $\ADS$ to linear orders of type $\omega + \omega^*$, is not Weihrauch reducible to $\ADC$. In this connection, we define a more natural stable form of $\ADS$ that we call $\GeneralSADS$, which is the restriction of $\ADS$ to linear orders of type $k + \omega$, $\omega + \omega^*$, or $\omega + k$, where $k$ is a finite number. We define $\GeneralSADC$ analogously. We prove that $\GeneralSADC$ is not Weihrauch reducible to $\SADS$, and so in particular, each of $\SADS$ and $\SADC$ is strictly weaker under Weihrauch reducibility than its general version. Finally, we turn to the principle $\CAC$, which asserts that every partial order on $\omega$ has an infinite chain or antichain. This has two previously studied stable variants, $\SCAC$ and $\WSCAC$, which were introduced by Hirschfeldt and Jockusch \cite{HS-2007}, and by Jockusch, Kastermans, Lempp, Lerman, and Solomon \cite{JKLLS-2009}, respectively, and which are known to be equivalent over $\RCA_0$. Here, we show that $\SCAC$ is strictly weaker than $\WSCAC$ under even computable reducibility.
\end{abstract}

\section{Introduction}

In the quest to understand the logic strength of Ramsey's theorem for pairs, initiated by Jockusch \cite{Jockusch-1972}, a myriad of related combinatorial principles were introduced and studied in their own right, giving rise to what is now called the reverse mathematics zoo \cite{Dzhafarov-zoo}. Two early examples, introduced by Hirschfeldt and Shore \cite{HS-2007}, were the {ascending/descending sequence principle} ($\ADS$) and the {chain/antichain principle} ($\CAC$). $\ADS$ asserts that every linear order (on $\omega$) has an infinite ascending or descending sequence, while $\CAC$ asserts that every partial order (on $\omega$) has an infinite chain or antichain. (See Section \ref{sec:bkg} for formal definitions.) While these principles have thus far been analyzed from the point of view of reverse mathematics, in this article we study them using the more nuanced framework of Weihrauch reducibility, which we describe below. We refer the reader to Soare~\cite{Soare-TA} and Simpson~\cite{Simpson-2009} for general background on computability and reverse mathematics, respectively, and to Hirschfeldt~\cite[Sections 6 and 9]{Hirschfeldt-2014} for a comprehensive survey of reverse mathematical results about Ramsey's theorem and other combinatorial problems.

As is well-known, there is a natural interplay between computability theory and reverse mathematics, with each of the benchmark subsystems of second-order arithmetic broadly corresponding to a particular level of computability-theoretic complexity (see, e.g.,~\cite[Section 1]{HS-2007} for details). In fact, this connection is deeper. The majority of principles one considers in reverse mathematics, like Ramsey's theorem, have the syntactic form
\[
	\forall X\, (\Phi(X) \to \exists Y \, \Psi(X,Y)),
\]
where $\Phi$ and $\Psi$ are arithmetical predicates. It is common to call such a principle a \emph{problem}, and to call each $X$ such that $\Phi(X)$ holds an \emph{instance} of this problem, and each $Y$ such that $\Psi(X,Y)$ holds a \emph{solution} to $X$. The instances of $\RT^n_k$ are thus the colorings $c : [\omega]^n \to k$, and the solutions to any such $c$ are the infinite homogeneous sets for this coloring. Over $\RCA_0$, an implication between problems (say $\mathsf{Q} \to \mathsf{P}$) can in principle make multiple applications of the antecedent ($\mathsf{Q}$), or split into cases in a non-uniform way; however, in practice, most implications have a simpler shape. To discuss these, we use the following notions of reduction between problems:

\begin{definition}\label{def:reds}
	Let $\mathsf{P}$ and $\mathsf{Q}$ be problems.	
	\begin{enumerate}
		\item $\mathsf{P}$ is \emph{computably reducible} to $\mathsf{Q}$, written $\mathsf{P} \cred \mathsf{Q}$, if every instance $X$ of $\mathsf{P}$ computes an instance $\widehat{X}$ of $\mathsf{Q}$, such that if $\widehat{Y}$ is any solution to $\widehat{X}$ then there is a solution $Y$ to $X$ computable from $X \oplus \widehat{Y}$.
		\item $\mathsf{P}$ is \emph{strongly computably reducible} to $\mathsf{Q}$, written $\mathsf{P} \scred \mathsf{Q}$, if every instance $X$ of $\mathsf{P}$ computes an instance $\widehat{X}$ of $\mathsf{Q}$, such that if $\widehat{Y}$ is any solution to $\widehat{X}$ then there is a solution $Y$ to $X$ computable from $\widehat{Y}$.
		\item $\mathsf{P}$ is \emph{Weihrauch reducible} to $\mathsf{Q}$, written $\mathsf{P} \ured \mathsf{Q}$, if there are Turing functionals $\Phi$ and $\Delta$ such that if $X$ is any instance of $\mathsf{P}$ then $\Phi^X$ is an instance of $\mathsf{Q}$, and if $\widehat{Y}$ is any solution to $\Phi^X$ then $\Delta^{X \oplus \widehat{Y}}$ is a solution to $X$.
		\item $\mathsf{P}$ is \emph{strongly Weihrauch reducible} to $\mathsf{Q}$, written $\mathsf{P} \sured \mathsf{Q}$, if there are Turing functionals $\Phi$ and $\Delta$ such that if $X$ is any instance of $\mathsf{P}$ then $\Phi^X$ is an instance of $\mathsf{Q}$, and if $\widehat{Y}$ is any solution to $\Phi^X$ then $\Delta^{\widehat{Y}}$ is a solution to~$X$.
	\end{enumerate}
\end{definition}
\noindent All of these reductions express the idea of taking a problem, $\mathsf{P}$, and computably (even uniformly computably, in the case of $\ured$ and $\sured$) transforming it into another problem, $\mathsf{Q}$, in such a way that being able to solve the latter computably (uniformly computably) tells us how to solve the former. This is a natural idea, and indeed, more often than not an implication $\mathsf{Q} \to \mathsf{P}$ over $\RCA_0$ (or at least, over $\omega$-models of $\RCA_0$) is a formalization of some such reduction. The strong versions above may appear more contrived, since it does not seem reasonable to deliberately bar access to the instance of the problem one is working with. Yet commonly, in a reduction of the above sort, the ``backward'' computation from $\widehat{Y}$ to $Y$ turns out not to reference the original instance. Frequently, it is just the identity.

Let $\mathsf{P} \leq_\omega \mathsf{Q}$ denote that every $\omega$-model of $\mathsf{Q}$ is a model of $\mathsf{P}$. It is easy to see that the following implications hold:
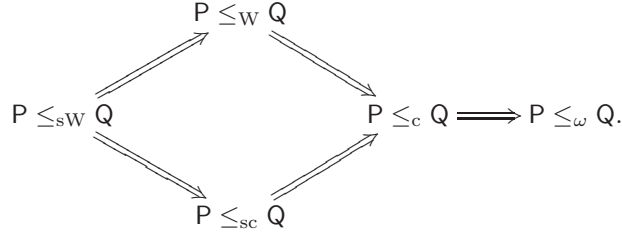
\[
\xymatrix{
& \mathsf{P} \ured \mathsf{Q} \ar@2[dr]\\
\mathsf{P} \sured \mathsf{Q} \ar@2[ur] \ar@2[dr] & & \mathsf{P} \cred \mathsf{Q} \ar@2[r] & \mathsf{P} \leq_\omega \mathsf{Q}.\\
& \mathsf{P} \scred \mathsf{Q} \ar@2[ur]
}
\]
No additional arrows can be added to this diagram (see~\cite[Section 1]{HJ-TA}). The notions of computable reducibility and strong computable reducibility were implicitly used in many papers on reverse mathematics, but were first isolated and studied for their own sake by Dzhafarov~\cite{Dzhafarov-2015}, and also form the basis of the iterated forcing constructions of Lerman, Solomon, and Towsner \cite{LST-2013}. Weihrauch reducibility (also called \emph{uniform reducibility}) and strong Weihrauch reducibility were introduced by Weihrauch~\cite{Weihrauch-1992}, under a different formulation than given above, and have been widely applied in the study of computable analysis. Later, these were independently rediscovered by Dorais, Dzhafarov, Hirst, Mileti, and Shafer~\cite{DDHMS-2016}, and shown to be the uniform versions of computable reducibility and strong computable reducibility, respectively (see~\cite[Appendix A]{DDHMS-2016}).

The investigation of these notions has seen a recent surge of interest, as evidenced, e.g., by \cite{BR-TA}, \cite{Dzhafarov-2015}, \cite{Dzhafarov-TA}, \cite{DPSW-TA}, \cite{FP-TA}, \cite{GHM-2015}, \cite{HJ-TA}, \cite{Patey-TA}. (A complete and updated bibliography is maintained by Brattka~\cite{Brattka-bib}.) Collectively, they provide a way of refining the analyses of effective and reverse mathematics, by elucidating subtler points of similarity and difference between various principles. In this paper, we apply this analysis to the above-mentioned principles $\ADS$, $\CAC$, and their variants. Specifically, we examine two natural formulations of the principle $\ADS$, which are equivalent from the classical viewpoint of reverse mathematics, but which we show to be different under Weihrauch reducibility. We then look at the so-called stable version of $\ADS$, first formulated by Hirschfeldt and Shore \cite{HS-2007}, and discover an overlooked form of this principle which is again classically equivalent, but different in the present setting. We conclude by examining two stable versions of $\CAC$, one formulated by Hirschfeldt and Shore, the other by Jockusch, Kastermans, Lempp, Lerman, and Solomon \cite{JKLLS-2009}, and show that, while these are known to be equivalent over $\RCA_0$, they are actually not equivalent under even computable reducibility.

The paper is structured as follows. In Section \ref{sec:bkg}, we formally define the principles we will be concerned with below, and discuss the basic relationships that hold between them. In Section \ref{sec:ads}, we prove our non-equivalence results about $\ADS$ and its stable variants. And in Section \ref{sec:cac}, we do the same for the two stable versions of $\CAC$. Our results are expressed in Figure \ref{fig:summary}, which appears in the next section.

\section{Background}\label{sec:bkg}

Throughout, we reserve $\leq$ for the natural ordering on $\omega$. All sets are assumed to be subsets of $\omega$, and all partial and linear orders are assumed to have field $\omega$ unless otherwise specified. As usual, if $\leq_P$ is a partial order, we write $x <_P y$ if $x \leq_P y$ and $x \neq y$.

\begin{definition}
	Let $\leq_L$ be a linear order.
	\begin{enumerate}
		\item An \emph{ascending sequence} for $\leq_L$ is a set $S \subseteq L$ such that for all $x,y \in S$, if $x \leq y$ then $x \leq_L y$.
		\item A \emph{descending sequence} for $\leq_L$ is a set $S \subseteq L$ such that for all $x,y \in S$, if $x \leq y$ then $y \leq_L x$.
		\item An \emph{ascending chain} for $\leq_L$ is a set $S \subseteq L$ such that for every $x \in S$ there are only finitely many $y \in S$ with $y \leq_L x$.
		\item A \emph{descending chain} for $\leq_L$ is a set $S \subseteq L$ such that for every $x \in S$ there are only finitely many $y \in S$ with $x \leq_L y$.
	\end{enumerate}
\end{definition}

The principle $\ADS$ below was formulated by Hirschfeldt and Shore~\cite[Sections 2 and 3]{HS-2007}. We formulate the analogues principle $\ADC$, which changes the formulation from sequences to chains.

\begin{highlight}{Ascending/descending sequence principle ($\ADS$)}
	Every linear order has an infinite ascending or descending sequence.
\end{highlight}

\begin{highlight}{Ascending/descending chain principle ($\ADC$)}
	Every linear order has an infinite ascending chain or descending chain.
\end{highlight}

Computably, there is no difference between these two principles, as we now show. The proof is straightforward, but we go through it carefully to highlight some of its features.

\begin{proposition}\label{prop:basic}
	\
	\begin{enumerate}
		\item $\ADC \sured \ADS$.
		\item $\ADS \cred \ADC$. In particular, $\ADS \cequiv \ADC$.
	\end{enumerate}
\end{proposition}

\begin{proof}
Clearly, every ascending sequence is an ascending chain, and every descending sequence is a descending chain. Hence, $\ADC \sured \ADS$, just via the identity functionals. This proves part 1. For part 2, fix an instance $\leq_L$ of $\ADS$. Let $S$ be any solution to $\leq_L$ as an instance of $\ADC$, say an infinite ascending chain. Then for every $x \in S$, almost all $y \in S$ satisfy $x <_L y$. Since $S$ is infinite, this means that for each $x \in S$ there is a $y \in S$ with $x <_L y$, and obviously, such a $y$ can be found $\leq_L \oplus S$-computably, uniformly in $x$. Iterating this procedure, $\leq_L \oplus S$ can computably thin $S$ out to an infinite ascending sequence $S' \subseteq S$ for $\leq_L$. Similarly, if $S$ is a descending chain, then $\leq_L \oplus S$ can computably thin $S$ to an infinite descending sequence for $\leq_L$. We conclude that $\ADS \cred \ADC$, via the identity functional in the forward direction, and the appropriate thinning procedure in the backward direction.
\end{proof}

Note that the reduction of $\ADS$ to $\ADC$ above is uniform modulo a single bit of information, namely, whether the $\ADC$-solution $S$ is an ascending chain or a descending chain. Otherwise, the reduction does not depend on the solution in any way. Thus, there are actually two fixed procedures such that from any $\ADC$-solution $S$, one or the other can be used to compute an $\ADS$-solution. This suggests that $\ADS$ is \emph{almost} uniformly reducible to $\ADC$. We show that this cannot be improved: $\ADS \nured \ADC$. Nonetheless, many of the results from \cite{HS-2007} work just as well whether we are working with sequences or chains, with some notable exceptions that we explore below.

Given a linear order $\leq_L$, we say $x$ is \emph{small} (in $\leq_L$), or $\leq_L$-small, if there are only finitely many $y$ with $y <_L x$, and we say $x$ is \emph{large} (in $\leq_L$), or $\leq_L$-large, if there are only finitely many $y$ with $x <_L y$.

\begin{definition}
	A linear order $\leq_L$ is \emph{stable} if every $x$ is either small or large.
\end{definition}

\noindent Note that every stable linear order has type $\omega + k$, $k + \omega^*$, or $\omega + \omega^*$, for some $k \in \omega$, depending as there are only finitely many large elements, only finitely many small elements, or infinitely many of each kind. Hirschfeldt and Shore \cite[Section 2]{HS-2007} formulated a version of $\ADS$ for orders of type $\omega + \omega^*$.

\begin{highlight}{Stable ascending/descending sequence principle ($\SADS$)}
	Every stable linear order $\leq_L$ with infinitely many small and large elements has an infinite ascending or descending sequence.
\end{highlight}

\noindent We define an analogous version of $\ADC$.

\begin{highlight}{Stable ascending/descending chain principle ($\SADC$)}
	Every stable linear order $\leq_L$ with infinitely many small and large elements has an infinite ascending chain or descending chain.
\end{highlight}

\noindent One would expect ``stable versions'' of $\ADS$ and $\ADC$ to be formulated for all stable linear orders, rather than just those of type $\omega + \omega^*$. However, it is easy to see that every computable linear order of type $\omega + k$ or $k + \omega^*$ for some finite $k$ has a computable solution (a computable infinite ascending or descending sequence, respectively). Thus, in the traditional framework of reverse mathematics, the restriction to orders of type $\omega + \omega^*$ is inconsequential. We can expect this not to be the case under uniform reducibility, and so formulate the following general versions of $\SADS$ and $\SADC$.

\begin{highlight}{Generalized $\SADS$ ($\GeneralSADS$)}
	Every stable linear order $\leq_L$ has an infinite ascending or descending sequence.
\end{highlight}

\begin{highlight}{Generalized $\SADC$ ($\GeneralSADC$)}
	Every stable linear order $\leq_L$ has an infinite ascending chain or descending chain.
\end{highlight}

Related to the above principles are the following well-known versions of Ramsey's theorem. For a set $X$, let $[X]^2$ denote the set of all ordered pairs $\seq{x,y} \in X^2$ with $x < y$. For $k \geq 1$, a \emph{stable $k$-coloring of pairs} is a map $c : [\omega]^2 \to k = \{0,\ldots,k-1\}$ with the property that $\lim_y c(x,y)$ exists for each $x$, which means there is a $j < k$ such that $c(x,y) = j$ for all sufficiently large $y$. Here and throughout, we write $c(x,y)$ in place of $c(\seq{x,y})$. A set $H$ is \emph{homogeneous} for $c$ if there is a $j < k$ such that $c(x,y) = j$ for all $\seq{x,y} \in [H]^2$, in which case we also say $H$ is homogeneous \emph{with color $j$}. A set $L$ is \emph{limit homogeneous} for $c$ if there is a $j < k$ such that $\lim_y c(x,y) = j$ for all $x \in H$, in which case we also say $L$ is limit homogeneous \emph{with color $j$}. Note that every infinite homogeneous set is limit homogeneous, with the same color $j$.

\begin{highlight}{Stable Ramsey's theorem for pairs ($\SRT^2$)}
	For every $k \geq 1$, every stable $k$-coloring of pairs has an infinite homogeneous set.
\end{highlight}

\begin{highlight}{$\Delta^0_2$ principle ($\D^2$)}
	For every $k \geq 1$, every stable $k$-coloring of pairs has an infinite limit homogeneous set.
\end{highlight}

\noindent The above principles were defined by Cholak, Jockusch, and Slaman \cite[Section 7]{CJS-2001}, and shown to be equivalent over $\RCA_0$ by Chong, Lempp, and Yang \cite[Theorem 1.4]{CLY-2010}. It is easy to check that the two are computable equivalent, and obviously, we even have $\D^2 \sured \SRT^2$. By contrast, Dzhafarov \cite[Corollary 3.3]{Dzhafarov-TA} has shown that $\SRT^2 \nured \D^2$.

The following proposition lists the basic relationships between all the above principles.

\begin{proposition}\label{prop:stablebasic}
	\
	\begin{enumerate}
		\item $\GeneralSADS \sured \SRT^2$ and $\GeneralSADC \sured \D^2$.
		\item $\GeneralSADS \sured \ADS$ and $\GeneralSADC \sured \ADC$.
		\item $\SADS \sured \GeneralSADS$ and $\SADC \sured \GeneralSADC$.
		\item $\SADC \sured \SADS$ and $\GeneralSADC \sured \GeneralSADS$.
		\item $\GeneralSADS \cred \SADC$.\\In particular, $\GeneralSADC \cequiv \GeneralSADS \cequiv \SADC \cequiv \SADS$.
	\end{enumerate}	
\end{proposition}

\begin{proof}
	For part 1, fix a stable linear order $\leq_L$. Define a coloring $c : [\omega]^2 \to 2$ as follows:
	\[
	c(x,y) =
	\begin{cases}
		0 \text{ if } y <_L x\\
		1 \text{ if } x <_L y.
	\end{cases}
	\]
	for all $x < y$. By stability of $\leq_L$, it follows that $c$ is a stable coloring. Now it is easy to see that if $H$ is homogeneous for $c$, then it is an ascending or descending sequence for $\leq_L$, and if $L$ is limit homogeneous for $c$, then it is an ascending or descending chain for $\leq_L$. Parts 2 and 3 are obvious. Part 4 is proved just like the first part of Proposition \ref{prop:basic}, while part 5 is proved just like the second part of Proposition \ref{prop:basic}.
\end{proof}




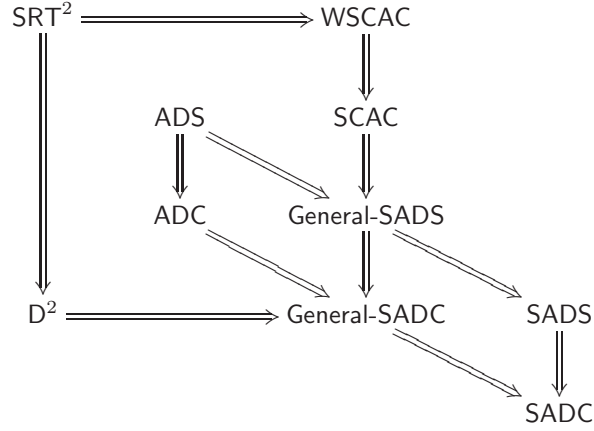
\begin{figure}
\[
\xymatrix{
\SRT^2 \ar@2[rr] \ar@2[ddd]& & \WSCAC \ar@2[d]\\
& \ADS \ar@2[dr] \ar@2[d]& \SCAC \ar@2[d]\\
& \ADC \ar@2[dr] & \GeneralSADS \ar@2[d] \ar@2[dr]\\
\D^2 \ar@2[rr] & & \GeneralSADC \ar@2[dr] & \SADS \ar@2[d]\\
& & & \SADC
}
\]
\caption{Relationships between principles under $\ured$. An arrow from one principle to another indicates that the latter is strongly Weihrauch reducible to the former.}\label{fig:summary}
\end{figure}

We can now state our main results about linear orders.

\begin{theorem}\label{thm:ads}
	$\SADS \nured \ADC$.	
\end{theorem}

\begin{theorem}\label{thm:d22}
	$\SADS \nured \D^2$.	
\end{theorem}

\begin{theorem}\label{thm:sads}
	$\GeneralSADC \nured \SADS$.	
\end{theorem}

Notice that, for the purposes of studying the above principles under Weihrauch reducibility, nothing is changed by considering linear orders on other infinite sets than just $\omega$. For if $(L,\leq_L)$ is a partial order and $L$ is infinite, we can uniformly $L$-computably build a bijection $f : \omega \to L$, and pass to the isomorphic order $\leq_{L'}$ on $\omega$ defined by $x \leq_{L'} y$ if and only if $f(x) \leq_L f(y)$. Then given an infinite ascending sequence/chain or descending sequence/chain $S$ for $\leq_{L'}$, $f(S)$ will be such a sequence/chain for $(L,\leq_L)$. Thus, restricting to orders with field $\omega$ is merely a notational convenience.

We now turn from linear orders to partial orders. Given a partial order $\leq_P$, we write $x \mid_P y$ if $x$ and $y$ are $\leq_P$-incomparable, i.e., if neither $x \leq_P y$ nor $y \leq_P x$ holds. We say $x$ is \emph{isolated} (in $\leq_P$), or $\leq_P$-isolated, if almost every $y$ is $\leq_P$-incomparable with $x$. Parts 1 and 2 of the following definition, and the subsequent principles $\SCAC$ and $\WSCAC$, are due to Hirschfeldt and Shore \cite[Definition 3.2]{HS-2007} and Jockusch, Kastermans, Lempp, Lerman, and Solomon \cite[Definitions 1.1 and 2.1]{JKLLS-2009}, respectively.

\begin{definition}
	A partial order $\leq_P$ is
	\begin{enumerate}
		\item \emph{stable} if either every $x$ is small or isolated, or else every $x$ is large or isolated;
		\item \emph{weakly stable} if every $x$ is small, isolated, or large.
	\end{enumerate}
\end{definition}

\begin{highlight}{Stable chain/antichain principle ($\SCAC$)}
	Every stable partial order has an infinite chain or antichain.
\end{highlight}

\begin{highlight}{Stable chain/antichain principle ($\WSCAC$)}
	Every weakly stable partial order has an infinite chain or antichain.
\end{highlight}

It was shown in \cite[Theorem 2.2]{JKLLS-2009} that over $\RCA_0$, the principles $\SCAC$ and $\WSCAC$ are equivalent. However, the proof of the non-trivial direction of this equivalence, that $\SCAC \to \WSCAC$, uses the antecedent, $\SCAC$, twice. We show that this is a necessary feature of the proof.

\begin{theorem}\label{thm:wscac}
	$\WSCAC \ncred \SCAC$.
\end{theorem}

\noindent It is tempting to ascribe this separation simply to the fact that while $\SCAC$ allows only two kinds of limiting behaviors (either small and isolated, or large and isolated), $\WSCAC$ allows three (small, isolated, and large). However, this is a false intuition, as the core of the proof relies not just on the difference between the numbers of limiting behaviors, but also in an essential way on their combinatorial properties. In that sense, this result differs significantly from Patey's recent result that, say, $\RT^2_3 \ncred \RT^2_2$ (\cite{Patey-TA}, Corollary 3.15). Indeed, $\WSCAC$ \emph{is} computably reducible (even strongly Weihrauch reducible) to $\RT^2_2$ (even $\SRT^2_2$). (See, e.g., \cite[Proposition 3.3]{HS-2007}.)

We summarize the relationships between the principles mentioned above in Figure \ref{fig:summary}. The following corollary of our results shows that no additional relationships can be added to the diagram.

\begin{corollary}
	No additional arrows can be added to Figure \ref{fig:summary}.
\end{corollary}

\begin{proof}
	First, we show that no arrows pointing to $\SRT^2$ or $\D^2$ can be added. As mentioned above, that $\SRT^2 \nured \D^2$ is shown in \cite[Corollary 3.3]{Dzhafarov-TA}. For the other possible arrows, it suffices to show that $\D^2 \nured \ADS$ and $\D^2 \nured \WSCAC$. These follow by Corollaries 2.29 and 3.12 in Hirschfeldt and Shore \cite{HS-2007}. The former gives an $\omega$-model of $\ADS$ in which $\SRT^2$ (and hence also $\D^2$, since the two are computably equivalent) fails. The latter gives an $\omega$-model of $\CAC$ (and hence of $\WSCAC$) in which $\SRT^2$ fails.
	
	Next, we show that no arrows pointing to $\ADS$ or $\ADC$ can be added. Hirschfeldt and Shore \cite[Proposition 2.10]{HS-2007} showed that over $\RCA_0$, $\ADS$ implies the so-called cohesive principle, $\COH$ (see \cite[Section 1]{HS-2007} for the definition), and it is easy to check that their proof actually shows that $\COH \sured \ADC$. On the other hand, Dzhafarov \cite[Corollary 4.5]{Dzhafarov-2015} showed that $\COH \nured \SRT^2$. Hence, $\ADC \nured \SRT^2$. That $\ADS \nured \ADC$ follows by Theorem \ref{thm:ads}. For the other possible arrows, it suffices to show that $\ADC \nured \WSCAC$. The desired witness of $\ADC$ is any computable linear order with no low infinite ascending chain or descending chain, which exists by \cite[Proposition 2.11]{HS-2007}. By contrast, by \cite[Corollary 3.5]{HS-2007}, there is an $\omega$-model of $\SCAC$ (and hence of $\WSCAC$, since the two are equivalent over $\omega$-models) consisting entirely of low sets.
	
	Finally, we show that no arrows pointing to any of $\WSCAC$, $\SCAC$, $\GeneralSADS$, $\GeneralSADC$, or $\SADS$ can be added. By Theorem \ref{thm:ads}, $\SADS \nured \ADC$, and by Theorem \ref{thm:d22}, $\SADS \nured \D^2$. In particular, $\SADS \nured \GeneralSADC$, and no arrow can be added pointing to $\SCAC$. By Theorem \ref{thm:wscac}, we have in particular that $\WSCAC \nured \SCAC$, so also no arrow can be added to $\WSCAC$. And by Theorem \ref{thm:sads}, $\GeneralSADC \nured \SADS$, which dispenses with the remaining arrow.
\end{proof}

\section{Linear orders}\label{sec:ads}

In this section, we prove Theorems \ref{thm:ads}, \ref{thm:d22}, and \ref{thm:sads}.

\subsection{Preliminaries}\label{S:defn}

We assume familiarity with forcing in arithmetic (see, e.g., \cite[Chapter 3]{Shore-TA} for an overview). Throughout, \emph{generic} (with respect to a fixed forcing notion) will mean arithmetically generic.

In what follows, let $\FinLin$ be the set of all linear orders on initial segments of $\omega$. For $\lambda \in \FinLin$, let $\leq_\lambda$ denote its ordering relation, and let $|\lambda|$ denote the largest $n$ such that $\leq_\lambda$ orders $\omega \res n$. Thus,
\[
	\lambda = (\omega \res |\lambda|,~\leq_\lambda).
\]
We code members of $\FinLin$ by their canonical indices, so that the map $\lambda \mapsto |\lambda|$ is computable. We say a linear order $(L,\leq_L)$ (on $\omega$ or an initial segment of $\omega$) \emph{extends $\lambda$} if $\omega \res |\lambda| \subseteq L$ and for all $x,y < |\lambda|$, we have $x \leq_\lambda y$ if and only if $x \leq_L y$. Note, if $\mu \in \FinLin$ extends $\lambda$ then $|\mu| \geq |\lambda|$.

For a Turing functional $\Delta$ and a set $X$, we adopt the convention that if $\Delta^X(x)$ is run for $u$ steps, the computation only queries the oracle about numbers $\seq{x,y}$ with $x,y < u$. For $\lambda \in \FinLin$, we write $\Delta^\lambda(x)$ to mean that the computation is run for only $|\lambda|$ steps with $\lambda$ as an oracle. Thus, if $\Delta^\lambda(x)\!\downarrow$, then for every $\mu \in \FinLin$ extending $\lambda$, we have that $\Delta^\mu(x) \downarrow = \Delta^\lambda(x)$; similarly, for any linear order $\leq_L$ on $\omega$ that extends $\lambda$, we have that $\Delta^{\leq_L}(x) \downarrow = \Delta^\lambda(x)$.

\begin{definition}\label{defn:forcing}
	Let $\mathbb{L}$ be the following notion of forcing. A \emph{condition} is a pair $p = (\lambda^p, a^p)$ as follows:
	\begin{itemize}
		\item $\lambda^p \in \FinLin$;
		\item $a^p$ is a map $|\lambda^p| \to \{\small,\large\} \times (\omega \res |\lambda^p|+1)$;
		\item if $y \leq_{\lambda^p} x$ and $a^p(x) = (\small,t)$, then $y < t$ and $a^p(y) = (\small,u)$ for some $u$;
		\item if $x \leq_{\lambda^p} y$ and $a^p(x) = (\large,t)$, then $y < t$ and $a^p(y) = (\large,u)$ for some $u$.
	\end{itemize}
	A condition $q$ \emph{extends} $p$, written $q \leq_{\mathbb{L}} p$, if $\lambda^q$ extends $\lambda^p$ and $a^q \supseteq a^p$.
\end{definition}

The idea here is that $a^p$ represents an assignment of each $x < |\lambda^p|$ to either the set of small or large elements of a stable linear order being approximated by $\lambda^p$. Specifically, if $a^p(x) = (\small,t)$ then for all $y < |\lambda^p|$ with $y \geq t$ we must have $x \leq_{\lambda^p} y$, while if $a^p(x) = (\large,t)$ then for all such $y$ we must have $y \leq_{\lambda^p} x$. We say $x$ is \emph{$p$-small} if $a^p(x) = (\small,t)$ for some $t$, and \emph{$p$-large} if $a^p(x) = (\large,t)$ for some $t$.

It is easy to see that any generic filter $\mathcal{F}_{\mathbb{P}}$ on $\mathbb{L}$ gives rise to a linear order of type $\omega + \omega^*$, given by $\bigcup_{p \in \mathcal{F}_{\mathbb{P}}} \lambda^p$. We denote this order by $G = (\omega,\leq_G)$, and use this also as a name for the generic order in the $\mathbb{L}$ forcing language.

If $p$ is a condition, we say a linear order $(L,\leq_L)$ (on $\omega$ or an initial segment of $\omega$) \emph{respects $p$} if $\leq_L$ extends $\lambda^p$ and, for all $x < |\lambda^p|$ and all $y \in L$, if $a^p(x) = (\small,t)$ and $y \geq t$ then $x \leq_L y$, and if $a^p(x) = (\large,t)$ and $y \geq t$ then $y \leq_L x$. Note that if $q$ extends $p$, then $\lambda^q$ respects $p$.

Definition~\ref{defn:forcing} ensures that if $x < |\lambda|$, then $a^p(x) = (\small,t)$ or $a^p(x) = (\large,t)$ for some $t \leq |\lambda^p|$. Thus, if $\leq_L$ respects $p$ and $x < |\lambda^p|$ is $p$-small, $x \le_L y$ for all $y \in L$ with $y \ge |\lambda^p|$; similarly, if $x < |\lambda|$ is $p$-large, $y \le_L x$ for all $y \in L$ with $y \ge |\lambda^p|$. (In other words, if $\le_L$ respects $p$, all elements of $L$ not already in $\lambda^p$ will be $\le_L$-above all $p$-small elements of $\lambda^p$ and $\le_L$-below all $p$-large elements of $\lambda^p$.)

The next lemma establishes that if $\lambda \in \FinLin$ respects $p$, then there is a condition $q \leq_{\mathbb{L}} p$ with $\lambda^q = \lambda$.

\begin{lemma}\label{lem:ext}
	Let $p$ be a condition. If $\lambda \in \FinLin$ respects $p$ then
	there are $q_0,q_1 \leq_{\mathbb{L}} p$ such that $\lambda^{q_i} = \lambda$ and every $z$ with $|\lambda^p| \leq z < |\lambda|$ is $q_0$-small and $q_1$-large.
\end{lemma}

\begin{proof}
	Fix $i \in \{0,1\}$. Define $a_i : |\lambda| \to \{\small,\large\} \times (\omega \res |\lambda|+1)$ as follows. For $z < |\lambda^p|$, define $a_i(z) = a^p(z)$. For $z$ with $|\lambda^p| \leq z < |\lambda|$, define $a_i(z) = (\small,|\lambda|)$ if $i = 0$, and define $a_i(z) = (\large,|\lambda|)$ if $i = 1$. Now let $q_i = (\lambda,a_i)$. We claim this is a condition, whence it follows that $q_i \leq_{\mathbb{L}} p$. It suffices only to verify the last two clauses in Definition~\ref{defn:forcing}. First, suppose $y \leq_{\lambda^{q_i}} x$ and $a^{q_i}(x) = (\small,t)$ for some $t$. We must show that $y < t$ and $y$ is $q_i$-small. We break into the following cases.
	
	\begin{highlight}{Case 1:}
		$x,y < |\lambda^p|$. In this case, we also have $y \leq_{\lambda^p} x$, since $\lambda^{q_i} = \lambda$ extends $\lambda^p$. Thus, $y < t$ and $y$ is $p$-small. By definition of $a^{q_i} = a_i$, this also means $y$ is $q_i$-small.
	\end{highlight}
	
	\begin{highlight}{Case 2:}
		$x < |\lambda^p|$ and $|\lambda^p| \leq y < |\lambda^q|$. By definition of $a^{q_i}$, we must have $a^p(x) = (\small,t)$, so $t \leq |\lambda^p| \leq y$. But then we cannot have $y \leq_{\lambda^{q_i}} x$, since $\lambda^{q_i}$ respects $p$. Thus, this case cannot obtain.
	\end{highlight}
	
	\begin{highlight}{Case 3:}
		$y < |\lambda^p|$ and $|\lambda^p| \leq x < |\lambda^q|$. Since $x \geq |\lambda^p|$, we must have that $t = |\lambda| = |\lambda^{q_i}|$, and since $|\lambda^p| \leq |\lambda^{q_i}|$, we have $y < t$. If $y$ were $p$-large, then we could not have $y \leq_{\lambda^{q_i}} x$ since $\lambda^{q_i}$ respects $p$, so $y$ must be $p$-small, and hence also $q_i$-small.
	\end{highlight}
	
	\begin{highlight}{Case 4:}
		$|\lambda^p| \leq x,y < |\lambda^q|$. Again, we must have $t = |\lambda^{q_i}|$, so certainly $y < t$. Since $x,y \geq |\lambda^{q_i}|$, $x$ and $y$ are either both $q_i$-small or both $q_i$-large, assuming $i = 0$ or $i = 1$ respectively. Hence, $y$ must be $q_i$-small.
	\end{highlight}
	
	\noindent We can similarly verify that if $x \leq_{\lambda^{q_i}} y$ and $a^{q_i}(x) = (\large,t)$, then $y < t$ and $y$ is $q_i$-large. This completes the proof.
\end{proof}

\begin{proposition}\label{prop:nocomp}
	If $G = (\omega,\leq_G)$ is the linear order of type $\omega + \omega^*$ given by a generic filter on $\mathbb{L}$, then $G$ has no $G$-computable infinite ascending or descending sequence.
\end{proposition}

\begin{proof}
	Fix a condition $p$ and a Turing functional $\Delta$, and suppose $p$ forces that $\Delta^G$ is total and infinite. We show there is a $q \leq_{\mathbb{L}} p$ forcing that $\Delta^G$ is not an ascending or descending sequence. Fix $\lambda \in \FinLin$ that respects $p$ such that there are numbers $x,y$ with $|\lambda^p| \leq x < y < |\lambda|$ and $\Delta^{\lambda}(x) \downarrow = \Delta^{\lambda}(y) \downarrow = 1$. Let $q_0,q_1$ be the extensions of $p$ given by Lemma~\ref{lem:ext}. If $x <_\lambda y$, let $q = q_1$. Then $a^q(x) = (\large,|\lambda|)$, so $q$ forces that $z \leq_G x$ for all $z \geq |\lambda|$. Suppose $q$ has an extension $r$ forcing that $\Delta^{G}$ is an infinite ascending or descending sequence for $G$. Since $r$ forces that $x,y \in \Delta^{G}$ and $x <_G y$, it must consequently force that $\Delta^{G}$ is an infinite ascending sequence, as $x < y$. But then $r$ must also force that there is a $z > |\lambda|$ in $\Delta^{G}$ with $x \leq_G z$, which is a contradiction. Hence, there can be no such extension $r$ of $q$, meaning $q$ forces that $\Delta^{G}$ is not an infinite ascending or descending sequence for $G$. If instead $y \leq_\lambda x$, we let $q = q_0$, and the argument is analogous.
\end{proof}

By contrast, we have the following basic fact about ``unbalanced'' linear orders, the proof of which is left to the reader.

\begin{proposition}\label{prop:infsubset}
	Let $\leq_L$ be a linear order and $X$ an infinite set.
	\begin{enumerate}
		\item If $\leq_L$ has no infinite ascending chain contained in $X$, then it has an \mbox{$\le_L \oplus X$}-computable infinite descending sequence contained in $X$.
		\item If $\leq_L$ has no infinite descending chain contained in $X$, then it has an \mbox{$\le_L \oplus X$}-computable infinite ascending sequence contained in $X$.
	\end{enumerate}
\end{proposition}

\noindent (Note that if $\leq_L$ has no infinite ascending/descending chain contained in $X$, then it also has no infinite ascending/descending sequence contained in $X$.)

\subsection{Proofs of the theorems}

In what follows, if $F$ is a finite set and $X$ is a non-empty set, we write $F < X$ or $F \leq X$ if $\max F < \min X$ or $\max F \leq \min X$, respectively. For $n \in \omega$, we write $n < X$ or $n \leq X$ if $\{n\} < X$ or $\{n\} \leq X$.

\begin{definition}\label{d:seetapun}
	Let $\leq_L$ be a linear order, and $\Psi$ a functional.
	\begin{enumerate}
		\item A finite set $F$ is an \emph{ascending blob} (respectively, \emph{descending blob}) if there exist $x < y$ such that $x <_L y$ (respectively, $y <_L x$) and $\Psi^{\leq_L \oplus F}(x) \downarrow = \Psi^{\leq_L \oplus F}(y) \downarrow = 1$. We call $\seq{x,y}$ the \emph{witness} for $F$.
		\item If $F_0 < F_1 < \cdots$ is an infinite sequence of ascending blobs (respectively, descending blobs), the \emph{ascending Seetapun tree} (respectively, \emph{descending Seetapun tree}) generated by this sequence is the set of all $\alpha \in \omega^{<\omega}$ such that $\alpha(i) \in F_i$ for all $i$, and there is no ascending blob (respectively, descending blob) $F \subseteq \ran(\alpha \res |\alpha| - 1)$.
	\end{enumerate}
\end{definition}

Note that in either the ascending or descending case, the Seetapun tree is a finitely branching tree, and if $\alpha$ is a node in it then $\alpha(i) < \alpha(i+1)$ for all $i < |\alpha|+1$. Thus, if $P$ is any infinite path through this tree, $\ran(P)$ is infinite. Note also that if $\alpha$ is a terminal node in an ascending Seetapun tree (respectively, descending Seetapun tree), then there is an ascending blob (respectively, descending blob) $F \subseteq \ran(\alpha)$. So no infinite path $P$ through this tree has any such blob in its range.

\begin{lemma}\label{lem:ascdescblobs}
	Let $\leq_L$ be a linear order, $\Psi$ a functional, and $X$ an infinite set.
	\begin{enumerate}
		\item Either there is an infinite sequence $F_0 < F_1 < \cdots$ of ascending blobs contained in $X$, and the ascending Seetapun tree generated by this sequence is finite, or there is an infinite set $Y \subseteq X$ that contains no ascending blob.
		\item Either there is an infinite sequence $F_0 < F_1 < \cdots$ of descending blobs contained in $X$, and the descending Seetapun tree generated by this sequence is finite, or there is an infinite set $Y \subseteq X$ that contains no descending blob.
	\end{enumerate}
\end{lemma}

\begin{proof}
	We prove part 1, the proof of part 2 being symmetric. Suppose it is not the case that there is an infinite sequence of ascending blobs contained in $X$ such that the ascending Seetapun tree generated by this sequence is finite. Then there are two cases to consider. First, suppose there is no infinite sequence $F_0 < F_1 < \cdots$ of ascending blobs contained in $X$. Then for some $n$, there can be no ascending blob $F \subseteq X$ with $n < \min F$. Thus, we can let $Y = X - n$. Second, suppose there is an infinite sequence $F_0 < F_1 < \cdots$ of ascending blobs, but the ascending Seetapun tree generated by it is infinite. In this case, choose any path $P$ through this tree, and let $Y = \ran(X)$. Since $P(i) \in F_i \subseteq X$ for all $i$, we have $Y \subseteq X$, and by definition, there is no ascending blob contained in $Y$.
\end{proof}

We can now prove our first main theorem. While we could give a simpler proof here, more in the style of that of Theorem \ref{thm:sads} below, the one we give is only slightly more intricate, and has the advantage of setting up the more involved proof of Theorem \ref{thm:d22}.

\begin{highlight}{Theorem~\ref*{thm:ads}.}
	$\SADS \nured \ADC$.	
\end{highlight}

\begin{proof}
	Fix functionals $\Phi$ and $\Psi$. We build a linear ordering $\leq_M$ of type $\omega + \omega^*$ to witness that $\SADS$ is not Weihrauch reducible to $\ADC$ via these functionals. If there is any linear order that $\Phi$ does not map to a linear order, we can just let $\leq_M$ be this order, and then we are done. We thus assume this is not the case. In particular, it must be forced in $\mathbb{L}$ that $\Phi^G$ is a linear order.

	If there is a condition $p$ forcing that $\Phi^G$ has a $G$-computable infinite ascending chain or descending chain $S$, let $G$ be any generic extension of $p$. By Proposition~\ref{prop:nocomp}, $G$ computes no infinite ascending or descending sequence for itself, so in particular, $\Psi^{G \oplus S}$ cannot define such a sequence. Thus, in this case, we can let $\leq_M$ be $\leq_G$, and again we are done. For the remainder of the proof, we may consequently assume it is forced that $\Phi^G$ is a linear order with no $G$-computable infinite ascending chain or descending chain; i.e., that there are no $G$-computable $\ADC$-solutions to $\Phi^G$.

	Fix any order $\leq_L$ of type $\omega + \omega^*$. We consider two cases.

	\begin{highlight}{Case 1:}
		there is an infinite sequence $F_0 < F_1 < \cdots$ of ascending blobs (or descending blobs) contained in $\omega$, and the ascending Seetapun tree (respectively, descending Seetapun tree) generated by this sequence is finite. Let us consider the ascending case, the descending case being symmetric. Call the ascending Seetapun tree $T$, and say it has height $n$, meaning $n = \max \{|\alpha| : \alpha \in T\}$. Let $U$ be the set of all strings $\beta \in \omega^{<\omega}$ with $|\beta| = n$ and $\beta(i) \in F_i$ for all $i < n$. Each $\beta \in U$ extends a terminal $\alpha \in T$, and so in particular, $\ran(\beta)$ contains some ascending blob $F_\beta$. Choose one blob $F_\beta$ for each $\beta \in U$, and designate witnesses $\seq{x,y}$ for each of these, along with $u_{\beta}$, the maximum use of $\Psi^{\leq_L \oplus F_{\beta}}(x)$ and $\Psi^{\leq_L \oplus F_{\beta}}(y)$; similarly, designate witnesses $\seq{x,y}$ for each of $F_0,\ldots,F_{n-1}$, along with $u_i$, the maximum use of $\Psi^{\leq_L \oplus F_i}(x)$ and $\Psi^{\leq_L \oplus F_i}(y)$. Let $W$ be the collection of all these witnesses, and let
		\[
			m = \max (\{u_i : i < n\} \cup \{u_\beta : \beta \in U\}) + 1.
		\]
		If $\seq{x,y} \in W$ then $x < y < m$ and $x <_L y$. Let $p$ be the condition with $\lambda^p$ equal to $\leq_L$ restricted to $\omega \res m$, and $a^p(z) = (\large,m)$ for all $z < m$; thus, if $\seq{x,y} \in W$, then $x <_{\lambda^p} y$ and $x$ is $p$-large. (In the descending case, we would have $y <_L x$, and would instead choose $p$ so that $a^p(z) = (\small,m)$ for all $z < m$.) 
		
		Let $G$ be any generic extension of $p$, and let $\leq_{\Phi^G}$ denote the ordering relation of $\Phi^G$. We claim, and prove below, that either there is an $i < n$ such that $F_i$ extends to an infinite ascending chain for $\Phi^G$, or there is a $\beta \in U$ such that $F_\beta$ extends to a descending chain for $\Phi^G$. Call this sequence $S$. Since $S$ begins with a blob ($F_i$~or~$F_\beta$) as an initial segment and $G$ agrees with $\leq_L$ up to the use of all witnesses, we have that $\Psi^{G \oplus S}$ contains $x$ and $y$ for some $\seq{x,y} \in W$. Since $\leq_G$ agrees with $\leq$ below $m$, this means that if $\Psi^{G \oplus S}$ is an infinite ascending or descending sequence for $G$, it must be ascending (as it contains two elements in increasing order). Since $x$ is large in $G$, this is impossible; therefore, $\Psi^{G\oplus S}$ cannot define an $\ADS$-solution for $G$. We can thus let $\leq_G$ be the desired linear order $\leq_M$.
		
		It thus remains to prove the claim. Let $\beta \in U$ be the string such that $\beta(i)$ is the $\leq_{\Phi^G}$-largest element of $F_i$ for each $i < n$. We consider two subcases.
		
		\begin{highlight}{Subcase a:}
			for some $i < n$, $\beta(i)$ is not large in $\Phi^G$; that is, the set $A = \{z : \beta(i) \leq_{\Phi^G} z\}$ is infinite. Then $A$ is an infinite $G$-computable set. Since (by assumption) $\Phi^G$ has no infinite $G$-computable descending chain, by Proposition~\ref{prop:infsubset}, $A$ must contain an infinite ascending chain $S'$. As $\beta(i)$ is the $\leq_{\Phi^G}$-largest element of $F_i$, we must have that $S = F_i \cup S'$ is an ascending chain as well.
		\end{highlight}
		
		\begin{highlight}{Subcase b:}
			otherwise. Let $u$ be the $\leq_{\Phi^G}$-smallest element of $F_\beta$. Since $u$ is large in $\Phi^G$ and $\Phi^G$ has infinite field, $B = \{z : z \leq_{\Phi^G} \beta(i)\}$ is an infinite $G$-computable set. Since $\Phi^G$ has no infinite $G$-computable ascending chain, by Proposition~\ref{prop:infsubset}, $B$ must contain an infinite descending chain $S'$. As $u$ is the $\le_{\Phi^G}$-smallest element of $F_\beta$, we must have that $S = F_\beta \cup S'$ is a descending chain as well.
		\end{highlight}
	\end{highlight}
	
	\begin{highlight}{Case 2:}
		otherwise. Since there is no infinite sequence $F_0 < F_1 < \cdots$ of ascending blobs contained in $\omega$ such that the ascending Seetapun tree generated by this sequence is finite, by part 1 of Lemma \ref{lem:ascdescblobs}, there is an infinite $Y_0 \subseteq \omega$ that contains no ascending blob. And since there is no infinite sequence $F_0 < F_1 < \cdots$ of descending blobs contained in $\omega$ such that the ascending Seetapun tree generated by this sequence is finite, there is in particular no such sequence of blobs contained in $Y_0$. So by part 2 of Lemma \ref{lem:ascdescblobs}, there is an infinite $Y_1 \subseteq Y_0$ that contains no descending blob. Now if $S \subseteq Y_1$ is any infinite ascending chain or descending chain for $\Phi^{\leq_L}$ such that $\Psi^{\leq_L \oplus S}$ defines a set, then this set can contain at most one element. Indeed, if not, then some sufficiently long initial segment of $S$ would be either an ascending or descending blob, which cannot be. Thus, in this case, we simply let $\leq_M$ be $\leq_L$.
	\end{highlight}
	This completes the proof.
\end{proof}

We turn next to proving Theorem \ref{thm:d22}, that $\SADS \nured \D^2$. The proof is very similar to that of Theorem \ref{thm:ads} above, but there are some combinatorial differences. The additional complexity is not in the actual construction, but rather in the definitions. In particular, we need an elaboration on Definition \ref{d:seetapun} and Lemma \ref{lem:ascdescblobs}. Given a finite, finitely branching tree $S \subseteq \omega^{<\omega}$, we let $\height(S)$ denote the height of $S$, so $\height(S) = \max \{|\beta| : \beta \in S\}$. We also let $\ran(S) = \{ \ran(\beta) : \beta \in S\}$. If $S'$ is another such tree, we write $S < S'$ if $\max \bigcup_{\beta \in S} \ran(\beta) < \min \bigcup_{\beta \in S'} \ran(\beta)$.

\begin{definition}\label{d:seetapuntrees}
	Let $\leq_L$ be a linear order, and $\Psi$ a functional.
	\begin{enumerate}
		\item A finitely branching well-founded tree $S \subseteq \omega^{<\omega}$ is an \emph{ascending tree-blob} (respectively, \emph{descending tree-blob}) if for each terminal $\beta \in S$ there exists an ascending blob (respectively, descending blob) $F \subseteq \ran(\beta)$.
		\item If $S_0 < S_1 < \cdots$ is an infinite sequence of ascending tree-blobs (respectively, descending tree-blobs), the \emph{ascending Seetapun tree} (respectively, \emph{descending Seetapun tree}) generated by this sequence is the set of all $\alpha \in \omega^{<\omega}$ such that $\alpha(i) \in \ran(S_i)$ for all $i$, and there is no ascending blob (respectively, descending blob) $F \subseteq \ran(\alpha \res |\alpha| - 1)$.
	\end{enumerate}
\end{definition}

Note that if $F$ is an ascending blob (respectively, descending blob), then the set of all initial segments of $F$ is an ascending tree-blob (respectively, descending tree-blob). Also, if $S_0 < S_1 < \cdots$ is an infinite sequence of ascending tree-blobs (respectively, descending tree-blobs) and the ascending Seetapun tree (respectively, descending Seetapun tree) generated by this sequence is finite, then this tree is itself an ascending Seetapun tree (respectively, descending Seetapun tree).

In either the ascending or descending case, if $S_0 < S_1 < \cdots$ is an infinite sequence of tree-blobs, and the Seetapun tree generated by this sequence is finite, say of height $n$, then any sequence of tree-blobs that begins with $S_0 < \cdots < S_{n-1}$ will generate the same Seetapun tree. Thus, in this case, we say the tree is generated by the finite sequence $S_0 < \cdots < S_{n-1}$.

\begin{definition}\label{d:forest}
	Fix $k \geq 2$. Let $\leq_L$ be a linear order, and $\Psi$ a functional. An \emph{ascending Seetapun $k$-forest} (respectively, \emph{descending Seetapun $k$-forest}) is a collection $\{S_{j,0} < \cdots < S_{j,t_j-1} : j < k\}$ as follows:
	\begin{itemize}
		\item $t_{k-1} = 1$, and for each $j < k-1$, $t_j = \sum_{i < t_{j+1}} \height(S_{j+1,i})$;
		\item for each $j < k$, $S_{j,0} < \cdots < S_{j,t_j-1}$ is a sequence of ascending tree-blobs (respectively, descending tree-blobs);
		\item for each $j < k-1$ and each $i$, $S_{j+1,i}$ is the ascending Seetapun tree (respectively, descending Seetapun tree), generated by $S_{j,t} < \cdots < S_{j,t + \height(S_{j+1,i})-1}$, where $t = \sum_{i' < i} \height(S_{j+1,i'})$.
	\end{itemize}
\end{definition}

\noindent In other words, $S_{1,0}$ is the Seetapun tree generated by
\[
	S_{0,0} < \cdots < S_{0,\height(S_{1,0})-1},
\]
$S_{1,1}$ is the Seetapun tree generated by
\[
	S_{0,\height(S_{1,0})} < \cdots < S_{0,\height(S_{1,0})+\height(S_{1,1})-1},
\]
and so on.

We say a tree-blob is \emph{contained} in a set $X$ if its range is, and we say a Seetapun $k$-forest $\{S_{j,0} < \cdots < S_{j,t_j-1} : j < k\}$ is \emph{contained} in $X$ if each $S_{j,i}$ is. We have the following analogue of Lemma \ref{lem:ascdescblobs}.

\begin{lemma}\label{lem:ascdesctreeblobs}
	Fix $k \geq 2$. Let $\leq_L$ be a linear order, $\Psi$ a functional, and $X$ an infinite set.
	\begin{enumerate}
		\item Either there is an ascending Seetapun $k$-forest contained in $X$, or there is an infinite set $Y \subseteq X$ that contains no ascending blob.
		\item Either there is a descending Seetapun $k$-forest contained in $X$, or there is an infinite set $Y \subseteq X$ that contains no descending blob.
	\end{enumerate}
\end{lemma}

\begin{proof}
	We prove part 1. First, if for some $n$, there is no ascending blob $F \subseteq X$ with $n < \min F$, we can let $Y = X - n$. So suppose not. Since, as noted above, every ascending blob can be regarded as a tree-blob, it follows that there exists an infinite sequence of ascending tree-blobs contained in $X$. Next, suppose there is an infinite such sequence such that the Seetapun tree generated by it is infinite. Then we can let $Y$ be the range of any infinite path through this tree, and argue as in the second case of Lemma \ref{lem:ascdescblobs}. So suppose also that this is not the case. We now construct a Seetapun $k$-forest contained in $X$ inductively, as follows. Let $S_{0,0} < S_{0,1} < \cdots$ be any infinite sequence of ascending tree-blobs contained in $X$. Having built, for some $j < k-1$, an infinite sequence $S_{j,0} < S_{j,1} < \cdots$ of tree-blobs in $X$, we define $S_{j+1,0} < S_{j+1,1} < \cdots$. Suppose $S_{j+1,i'}$ has been defined for all $i' < i$, and let $t$ be least such that $\bigcup_{i' < i} \ran(S_{j+1,i'}) < \min \ran(S_{j,t})$. (If $i = 0$, set $t = 0$.) As $S_{j,t} < S_{j,t+1} < \cdots$ is itself an infinite sequence of tree-blobs contained in $X$, the ascending Seetapun tree generated by it must, by assumption, be finite, and we le this be $S_{j+1,i}$. Once the sequence $S_{k-1,0} < S_{k-1,1} < \cdots$ is built, we define $t_{k-1} = 1$, and for each $j < k-1$, $t_j = \sum_{i < t_{j+1}} \height(S_{j+1,i})$. It is easy to see that $\{S_{j,0} < \cdots < S_{j,t_j-1} : j < k\}$ is indeed a Seetapun $k$-forest, as desired. The proof of part 2 is analogous.
\end{proof}

Our final lemma for proving Theorem \ref{thm:d22} will allow us to build limit homogeneous sets for stable colorings. It appears, along with a proof, as Lemma 2.6 in \cite{Dzhafarov-TA}.

\begin{lemma}\label{lem:monochr}
	Fix $k \geq 2$. Let $\leq_L$ be a linear order, $\Psi$ a functional, and $\{S_{j,0} < \cdots < S_{j,t_j-1} : j < k\}$ an ascending or descending Seetapun $k$-forest. If $c : [\omega]^2 \to k$ is a stable coloring, then there is a $j < k$ and an $i < t_j$ such that for some terminal $\beta \in S_{j,i}$, $\lim_y c(x,y) = j$ for all $x \in \ran(\beta)$.
\end{lemma}

We now give the proof of the theorem.

\begin{highlight}{Theorem~\ref*{thm:d22}}
	$\SADS \nured \D^2$.	
\end{highlight}

\begin{proof}
	Fix $\Phi$ and $\Psi$. We may assume that for some $k$, it is forced that $\Phi^G$ is a stable coloring $[\omega]^2 \to k$ with no $G$-computable infinite limit homogeneous set. By deleting colors from $\{0,\ldots,k-1\}$ if necessary and renaming the ones that remain, we may further assume that it is forced that for each $j < k$, there are infinitely many $x$ with $\lim_y \Phi^G(x,y) = j$. If it were the case that $k = 1$, then almost all $x$ would have the same limit under $\Phi^G$, so some co-initial segment of $\omega$ would be limit homogeneous for $\Phi^G$, which cannot be by assumption. Thus, $k \geq 2$.
	
	Fix any order $\leq_L$ of type $\omega + \omega^*$. We consider two cases.
	
	\begin{highlight}{Case 1:}
		there is an ascending Seetapun $k$-forest (or descending Seetapun $k$-forest).	Consider the ascending case, since the descending case is symmetric, and say the forest is $\{S_{j,0} < \cdots < S_{j,t_j-1} : j < k\}$. For each $j < k$, each $i < t_j$, and each terminal $\beta \in S_{j,i}$, there is then an ascending blob $F_\beta \subseteq \ran(\beta)$. Designate witnesses $\seq{x,y}$ for each such $F_\beta$, along with $u_\beta$, the maximum use of $\Psi^{\leq_L \oplus F_{\beta}}(x)$ and $\Psi^{\leq_L \oplus F_{\beta}}(y)$. Let $W$ be the collection of all these witnesses, and let $m$ be the maximum of all the uses $u_\beta$. For each $\seq{x,y} \in W$, let $p$ be the condition with $\lambda^p$ equal to $\leq_L$ restricted to $\omega \res m$, and $a^p(z) = (\large,m)$ for all $z < m$. Now let $G$ be any generic extension of $p$. By Lemma \ref{lem:monochr}, there is a $j < k$ and an $i < t_j$ such that for some terminal $\beta \in S_{j,i}$, $\lim_y \Phi^G(x,y) = j$ for all $x \in \ran(\beta)$. In particular, $F_\beta$ is limit homogeneous for $\Phi^G$, and by assumption, there are infinitely many elements that have the same limit under $\Phi^G$ as the elements of $F_\beta$. Thus, $F_\beta$ can be extended to an infinite limit homogeneous set $I$ for $\Phi^G$, with $x > m$ for all $x \in I - F_\beta$. Now as in the proof of Theorem \ref{thm:ads}, we conclude that $\Phi^{G \oplus I}$ cannot define an $\ADS$-solution for $G$, using the witnesses $x$ and $y$ for the fixed ascending blob $F_\beta$ and the fact that $x$ is large in $G$.
	\end{highlight}
	
	\begin{highlight}{Case 2:}
		otherwise. By part 1 of Lemma \ref{lem:ascdesctreeblobs}, there is an infinite $Y_0 \subseteq \omega$ that contains no ascending blob, and by part 2 of the same lemma, there is an infinite $Y_1 \subseteq Y_0$ that contains no descending blob. As in the proof of Theorem \ref{thm:ads}, if $I \subseteq Y_1$ is any infinite limit homogeneous set for $\Phi^{\leq_L}$ such that $\Psi^{\leq_L \oplus I}$ defines a set, then this set cannot be infinite. We thus let $\leq_M$ be $\leq_L$. \qedhere
	\end{highlight}
\end{proof}

We conclude by proving Theorem \ref{thm:sads}.

\begin{highlight}{Theorem~\ref*{thm:sads}.}
	$\GeneralSADC \nured \SADS$.
\end{highlight}

\begin{proof}
	Fix functionals $\Phi$ and $\Psi$. We build a stable linear ordering $\leq_M$ to witness that $\GeneralSADC$ is not Weihrauch reducible to $\SADS$ via these functionals. Analogously to Theorem \ref{thm:ads}, we may assume $\Phi$ takes every stable linear order to a linear order of type $\omega + \omega^*$.
	
	\begin{highlight}{Case 1:}
		there is a condition $p$ and an $n \in \omega$, for which there is no $q \leq p$ and no finite $F > n$	such that $q$ forces that $F$ is ascending under $\leq_{\Phi^G}$ and there is an $x > n$ with $\Psi^{G \oplus F}(x) \downarrow = 1$. In this case, let $G$ be any generic extension of $p$. Since $\Phi^G$ has order type $\omega + \omega^*$, it must have an infinite ascending sequence $S > n$. Then $\Psi^{G \oplus S}$ cannot define a non-empty, let alone an infinite, set.
	\end{highlight}
	
	\begin{highlight}{Case 2:}
		there is a condition $p$ and a finite set $F$ with the following properties:
		\begin{itemize}
			\item $p$ forces that $F$ is ascending under $\leq_{\Phi^G}$ and the $\leq_{\Phi^G}$-largest element of $F$ is small in $\leq_{\Phi^G}$;
			\item there is an $x$ such that $\Psi^{G \oplus F}(x) \downarrow = 1$.
		\end{itemize}
		By extending $p$ if necessary, we may assume that $p$ decides whether $x$ is small or large in $\leq_G$. If $p$ forces that $x$ is small, let $\leq_M$ be any linear order of type of $k + \omega^*$ that respects $p$. Then $x$ must also be small in $\leq_M$. Moreover, even though $M$ is not generic, $p$ forces that the $\leq_{\Phi^G}$-largest element, $u$, of $F$ is small in $\leq_{\Phi^G}$, and $u$ must also be small in $\Phi^{\leq_M}$. (Indeed, $p$ forces that there is a $z$ such that $u <_{\Phi^G} y$ for all $y > z$. Hence, there is a $z$ such that for every $y > z$ and every $q$ extending $p$, there is an $r$ extending $q$ forcing that $x <_{\Phi^G} y$, meaning that $\Phi^{\lambda^r}(x,y) \downarrow = 1$. Now if there were a $y > z$ such that $y <_{\Phi^M} x$ then we could take an initial segment $\lambda$ of $\leq_M$ extending $\lambda^p$ such that $\Phi^{\lambda}(x,y) \downarrow = 1$. By Lemma \ref{lem:ext}, we could then choose a $q \leq p$ with $\lambda^q = \lambda$, and for this $q$ no $r$ as above could exist, a contradiction.) Thus, $F$ is extendible to an infinite ascending sequence $S$ for $\Phi^{\leq_M}$, and $\Psi^{\leq_M \oplus S}(x) \downarrow = 1$. But since $\leq_M$ has order type $k + \omega^*$, it has no infinite ascending chain, and $x$ cannot be part of any infinite descending chain. Hence, $\Psi^{\leq_M \oplus S}$ cannot define a $\GeneralSADC$-solution for $\leq_M$. If $p$ instead forces that $x$ is large in $\leq_G$, we instead let $\leq_M$ be any linear order of type of $\omega + k$ that respects $p$, and argue similarly.
	\end{highlight}
	
	\begin{highlight}{Case 3:}
		otherwise. Let $\leq_G$ be any generic linear order of type $\omega + \omega^*$. The failure of Case 1 is a density fact. Since $G$ is generic, we can $G$-computably find a sequence of finite $\leq_{\Phi^G}$-ascending sets $F_0 < F_1 < \cdots$ and numbers $x_0 < x_1 < \cdots$ such that $\Psi^{G \oplus F}(x) \downarrow = 1$. Since Case 2 does not hold, the $\leq_{\Phi^G}$-largest element of each $F_i$ must be large in $\leq_{\Phi^G}$. Thus, we can $G$-computably pick out an increasing sequence $y_0 < y_1 < \cdots$ such that $y_0 >_{\Phi^G} y_1 >_{\Phi^G} \cdots$ (namely, $y_0$ is the $\leq_{\Phi^G}$-largest element of $F_0$, and given $y_i$, which is the $\leq_{\Phi^G}$-largest element of $F_j$, we search for the least $k > j$ such that the $\leq_{\Phi^G}$-largest element $y$ of $F_k$ satisfies $y <_{\Phi^G} y_i$, and we set $y_{i+1} = y$). Thus, the $y_i$ form a $G$-computable infinite descending sequence $S$ in $\leq_{\Phi^G}$. By Proposition \ref{prop:nocomp}, $G$ has no $G$-computable infinite ascending chain or descending chain, so in particular, $\Psi^{G \oplus S}$ cannot define such a sequence. \qedhere
	\end{highlight}
\end{proof}

\section{Partial orders}\label{sec:cac}

We now turn to our final result, Theorem \ref{thm:wscac}.

\subsection{Preliminaries}

We begin with a number of definitions that largely parallel those of Section \ref{S:defn} above. Let $\FinPO$ be the set of all partial orders on initial segments of $\omega$. For $\pi \in \FinPO$, let $\leq_\pi$ denote its ordering relation, and let $|\pi|$ denote the largest $n$ such that for all $x,y < n$, $\pi$ either orders $x$ and $y$ or declares them incomparable. We say a partial order $(P,\leq_P)$ \emph{extends $\pi$} if $\omega \res |\pi| \subseteq P$ and for all $x,y < |\pi|$, we have $x \leq_\pi y$ if and only if $x \leq_P y$. If $\rho \in \FinPO$ extends $\pi$ then $|\rho| \geq |\pi|$. We adopt the same use conventions for computations from members of $\FinPO$ as we did for computations from members of $\FinLin$.

\begin{definition}
	Let $\mathbb{P}$ be the following notion of forcing. A \emph{condition} is a pair $p = (\pi^p, a^p)$ as follows:
	\begin{itemize}
		\item $\pi^p \in \FinPO$;
		\item $a^p$ is a map $|\pi^p| \to \{\small,\large,\isolated\} \times (\omega \res |\pi^p|+1)$;
		\item if $y \leq_{\pi^p} x$ and $a^p(x) = (\small,t)$, then $y < t$ and $a^p(y) = (\small,u)$ for some $u$;
		\item if $x \leq_{\pi^p} y$ and $a^p(x) = (\large,t)$, then $y < t$ and $a^p(y) = (\large,u)$ for some $u$;
		\item if $x \leq_{\pi^p} y$ or $y \leq_{\pi^p} x$ and $a^p(x) = (\isolated,t)$, then $y < t$.
	\end{itemize}
	A condition $q$ \emph{extends} $p$, written $q \leq_{\mathbb{P}} p$, if $\pi^q$ extends $\pi^p$ and $a^q \supseteq a^p$.
\end{definition}

We define $x$ being \emph{$p$-small} and \emph{$p$-large} as for linear orders (and collectively call such elements \emph{$p$-non-isolated}), and say $x$ is \emph{$p$-isolated} if $a^p(x) = (\isolated,t)$ for some $t$. Obviously, any generic filter on $\mathbb{P}$ gives rise to a weakly stable partial order, which we denote by $G = (\omega,\leq_G)$. Going forward, we will refer to conditions in $\mathbb{P}$ explicitly as \emph{$\mathbb{P}$-conditions}, to avoid confusion with the notion $\mathbb{M}$ that we define below.

If $p$ is a $\mathbb{P}$-condition, we say a partial order $(P,\leq_P)$ (on $\omega$ or an initial segment of $\omega$) \emph{respects $p$} if $\leq_P$ extends $\pi^p$ and, for all $x < |\pi^p|$ and all $y \in P$, if $a^p(x) = (\small,t)$ and $y \geq t$ then $x \leq_P y$, if $a^p(x) = (\large,t)$ and $y \geq t$ then $y \leq_P x$, and if $a^p(x) = (\isolated,t)$ and $y \geq t$ then $x \mid_P y$.

We have the following analogues of Lemma \ref{lem:ext}, Proposition \ref{prop:nocomp}, and Proposition \ref{prop:infsubset} in the setting of partial orders.

\begin{lemma}\label{lem:partialext}
	Let $p$ be a $\mathbb{P}$-condition, and suppose $\pi \in \FinPO$ respects $p$.
	\begin{enumerate}
		\item There are $q_0,q_1,q_2 \leq_{\mathbb{P}} p$ such that $\pi^{q_i} = \pi$ and every $z$ with $|\pi^p| \leq z < |\pi|$ is $q_0$-small, $q_1$-isolated, and $q_2$-large.
		\item If $|\pi^p| \leq x,y < |\pi|$ and $x <_\pi y$, there are $r_0,r_1 \leq_{\mathbb{P}} p$ such that $\pi^{r_i} = \pi$, $x$ is $r_0$-small and $r_1$-isolated, and $y$ is $r_0$-isolated and $r_1$-large.
		\item If $|\pi^p| \leq x,y < |\pi|$ and $x$ and $y$ are $\leq_\pi$-incomparable, there are $s_0,s_1 \leq_{\mathbb{P}} p$ such that $\pi^{s_i} = \pi$, $x$ is $s_0$-small and $s_1$-isolated, and $y$ is $s_0$-isolated and $s_1$-large.
	\end{enumerate}
\end{lemma}

\begin{proof}
	Part 1 is proved just like Lemma \ref{lem:ext}. For part 2, fix $i \in \{0,1\}$. Define $a_i : |\pi| \to \{\small,\large,\isolated\} \times (\omega \res |\pi| + 1)$ as follows. For all $z < |\pi|$, let $a_i(z) = a^p(z)$, and for all $z$ with $|\pi^p| \leq z < |\pi|$, let
	\[
		a_i(z) =
		\begin{cases}
			(\small,|\pi|) & \text{if } z \ngeq_\pi y,\\
			(\isolated,|\pi|) & \text{if } z \geq_\pi y,
		\end{cases}
	\]
	if $i = 0$, and
	\[
		a_i(z) =
		\begin{cases}
			(\isolated,|\pi|) & \text{if } z \ngeq_\pi y,\\
			(\large,|\pi|) & \text{if } z \geq_\pi y,
		\end{cases}
	\]
	if $i = 1$. Let $r_i = (\pi,a_i)$. The verification that this is a $\mathbb{P}$-condition is now just as in Lemma \ref{lem:ext}. Part 3 is proved analogously.
\end{proof}

\begin{proposition}\label{prop:partialnocomp}
	If $G = (\omega,\leq_G)$ is the weakly stable partial order given by a generic filter on $\mathbb{P}$, then $G$ has no $G$-computable infinite chain or antichain.
\end{proposition}

\begin{proof}
	Just like Proposition \ref{prop:nocomp}.
\end{proof}

\begin{proposition}\label{prop:partialinfsubset}
	Let $\leq_P$ be a partial order and $X$ an infinite set.
	\begin{enumerate}
		\item If $\leq_P$ has no infinite chain contained in $X$, then it has an \mbox{$\leq_P \oplus X$}-computable infinite antichain contained in $X$.
		\item If $\leq_P$ has no infinite antichain contained in $X$, then it has an \mbox{$\le_P \oplus X$}-computable infinite chain contained in $X$.
	\end{enumerate}
\end{proposition}

\begin{proof}
	Just like Proposition \ref{prop:infsubset}.
\end{proof}

We now define an embellishment of $\mathbb{P}$ that will allow us to build chains and antichains for stable partial orders computable from $G$.

\begin{definition}\label{defn:Mforcing}
	Let $\mathbb{M}$ be the following notion of forcing.	 A \emph{condition} is a sequence $M$ consisting of a $\mathbb{P}$-condition $p^M$, a finite collection $R^M$ of Turing functionals, and for each $\Phi \in R^M$, a triple $(C^{M,\Phi},A^{M,\Phi},I^{M,\Phi})$ as follows:
	\begin{itemize}
		\item $p^M$ forces that $\Phi^G$ is a stable partial order with no $G$-computable infinite chain or antichain;
		\item $C^{M,\Phi}$ is a finite set, and $p^M$ forces that $C^{M,\Phi}$ is a chain in $\Phi^G$;
		\item $A^{M,\Phi}$ is a finite set, and $p^M$ forces that $A^{M,\Phi}$ is an antichain in $\Phi^G$;
		\item $I^{M,\Phi}$ is a computable infinite set;
		\item $C^{M,\Phi} < I^{M,\Phi}$, and $p^M$ forces that all $x \in C^{M,\Phi}$ are comparable under $\leq_{\Phi^G}$ with all $y \geq I^{M,\Phi}$;
		\item $A^{M,\Phi} < I^{M,\Phi}$, and $p^M$ forces that all $x \in A^{M,\Phi}$ are incomparable under $\leq_{\Phi^G}$ with all $y \geq I^{M,\Phi}$;
	\end{itemize}
	A condition $N$ \emph{extends} $M$, written $N \leq_{\mathbb{M}} M$, if $p^N \leq_{\mathbb{P}} p^M$, $R^N \supseteq R^M$, and for each $\Phi \in R^M$, $C^{M,\Phi} \subseteq C^{N,\Phi} \subseteq C^{M,\Phi} \cup I^{M,\Phi}$, $A^{M,\Phi} \subseteq A^{N,\Phi} \subseteq A^{M,\Phi} \cup I^{M,\Phi}$, and $I^{N,\Phi} \subseteq I^{M,\Phi}$.
\end{definition}

Thus, in an $\mathbb{M}$-condition $M$ with $\Phi \in R^M$, each of $(C^{M,\Phi},I^\Phi)$ and $(A^{M,\Phi},I^\Phi)$ is just a Mathias condition. (See, e.g., \cite{CDHS-2014} and \cite{CDS-TA} for background on Mathias forcing in computability theory.) A generic filter $\mathcal{F}_{\mathbb{M}}$ on $\mathbb{M}$ thus produces a weakly stable partial order $G = \bigcup_{M \in \mathcal{F}_{\mathbb{M}}} \pi^{p^M}$ which is generic for $\mathbb{P}$, a collection $R^G = \bigcup_{M \in \mathcal{F}_{\mathbb{M}}} R^M$ of Turing functionals such that $\Phi^G$ is stable partial order with no $G$-computable infinite chain or antichain, and for each $\Phi \in R^G$, a chain $C^{G,\Phi} = \bigcup_{M \in \mathcal{F}_{\mathbb{M}}} C^{M,\Phi}$ and antichain $A^{G,\Phi} = \bigcup_{M \in \mathcal{F}_{\mathbb{M}}} A^{M,\Phi}$. For the remainder of this section, let $\mathcal{F}_{\mathbb{M}}$ be fixed, and let the above generic objects be taken with respect to it. We also use $G$, $R^G$, $A^G$, and $C^G$ as names for these objects in the $\mathbb{M}$ forcing language.


\subsection{Proof of the theorem}

The proof will follow by the following sequence of lemmas. Note that if $p$ is a $\mathbb{P}$-condition forcing that every element of a finite set $F$ is $\leq_{\Phi^G}$-non-isolated or $\leq_{\Phi^G}$-isolated, then there is an $m$ such that $p$ forces that every $x \in F$ is $\leq_{\Phi^G}$-comparable with every $y \geq m$ or $\leq_{\Phi^G}$-incomparable with every such $y$. In what follows, we denote the least such $m$ by $m_{p,F}$.

\begin{lemma}\label{lem:Gaddfunctional}
	For each Turing functional $\Phi$, if $\Phi^G$ is a stable coloring with no $G$-computable infinite chain or antichain, then $\Phi \in R^G$.
\end{lemma}

\begin{proof}
	Let $M$ be any $\mathbb{M}$-condition such that $p^M$ forces that $\Phi^G$ is a stable partial order with no $G$-computable infinite chain or antichain. We define an $N \leq_{\mathbb{M}} M$ with $\Phi \in R^N$, which suffices, by genericity. If $\Phi \in R^M$, let $N = M$. Otherwise, let $R^N = R^M \cup \{\Phi\}$, $C^{N,\Phi} = A^{N,\Phi} = \emptyset$, and $I^{N,\Phi} = \omega$, and let the rest of the $N$ agree with $M$.
\end{proof}

\begin{lemma}\label{lem:infcompoents}
	For each $\Phi \in R^G$, each of $C^{G,\Phi}$ and $A^{G,\Phi}$ is infinite.
\end{lemma}

\begin{proof}
	Let $M$ be any $\mathbb{M}$-condition with $\Phi \in R^M$. We show there is an $N \leq_{\mathbb{M}} M$ with $|C^{N,\Phi}| = |C^{M,\Phi}| + 1$ and $|A^{N,\Phi}| = |A^{M,\Phi}| + 1$. If, for every $y \in I^{M,\Phi}$, every $p \leq_{\mathbb{P}} p^M$ forced that $y$ is $\leq_{\Phi^G}$-isolated, then $I^{M,\Phi}$ would be a computable infinite set of elements all of which are $\leq_G$-isolated (for the actual generic $G$), so $\Phi^G$ would heve a $G$-computable infinite antichain contained in $I^{M,\Phi}$. But this cannot be, since $p$ forces that $\Phi^G$ has no $G$-computable infinite antichain. Hence, there must be an $x_0 \in I^{M,\Phi}$ and a $p_0 \leq_{\mathbb{P}} p$ forcing that $x_0$ is $\leq_{\Phi^G}$-non-isolated. Similarly, there must be an $x_1 \in I^{M,\Phi}$ and a $p_1 \leq_{\mathbb{P}} p_0$ forcing that $x_1$ is $\leq_{\Phi^G}$-isolated. Let $p^N = p_1$, $C^{N,\Phi} = C^{M,\Phi} \cup \{x_0\}$, and $A^{N,\Phi} = A^{M,\Phi} \cup \{x_1\}$. Let $I^{N,\Phi} = \{x \in I^{M,\Phi} : x > m_{p^N,\{x_0,x_1\}}\}$, and let the rest of $N$ agree with $M$. Now $N$ is the desired extension of $M$.
\end{proof}

The next lemma presents the key diagonalization step for our proof.

\begin{lemma}\label{lem:maindiag}
	Fix $\Phi \in R^G$ and Turing functionals $\Gamma$ and $\Delta$. If each of $\Gamma^{G \oplus C^{G,\Phi}}$ and $\Delta^{G \oplus A^{G,\Phi}}$ is total and defines a chain or antichain for $G$, then one of the two defines a finite set.
\end{lemma}

\begin{proof}
	Fix a condition $M$ such that $\Phi \in R^M$, and such that $M$ forces (in $\mathbb{M}$) that each of $\Gamma^{G \oplus C^{G,\Phi}}$ and $\Delta^{G \oplus A^{G,\Phi}}$ is total, and for each of the two, either that it defines a chain for $G$, or that it defines an antichain. We exhibit an $N \leq_{\mathbb{M}} M$ forcing that either $\Gamma^{G \oplus C^{G,\Phi}}$ or $\Delta^{G \oplus A^{G,\Phi}}$ defines a finite set, which gives the lemma.
	
	Since $\Phi \in R^M$, $p^M$ forces that $\Phi^G$ is a stable partial order. Assume that $p^M$ forces that every number is either $\leq_{\Phi^G}$-small or $\leq_{\Phi^G}$-isolated. The case where $p^M$ forces that every number is either $\leq_{\Phi^G}$-large or $\leq_{\Phi^G}$-isolated is symmetric. We consider the following cases.
	
	\begin{highlight}{Case 1:}
		there is a $p \leq_{\mathbb{P}} p^M$, an infinite computable subset $I$ of $I^{M,\Phi}$, and an $n \in \omega$ such that for all $q \leq_{\mathbb{P}} p$ and all finite sets $F \subseteq I$, if $q$ forces that $F$ is a chain for $\Phi^G$ and all its elements are $\leq_{\Phi^G}$-small, then there is no $x \geq n$ with $\Gamma^{\pi^q \oplus (C^{M,\Phi} \cup F)}(x) \downarrow = 1$. In this case, define $N$ as follows. Let $p^N = p$ and $I^{N,\Phi} = I$, and let the rest of $N$ agree with $M$. Then $N \leq_{\mathbb{M}} M$, and clearly $N$ forces that the set defined by $\Gamma^{G \oplus C^{G,\Phi}}$ contains no numbers $x \geq n$.
	\end{highlight}
	
	\begin{highlight}{Case 2:}
		there is a $p \leq_{\mathbb{P}} p^M$, an infinite computable subset $I$ of $I^{M,\Phi}$, and an $n \in \omega$ such that for all $q \leq_{\mathbb{P}} p^M$ and all finite sets $F \subseteq I$, if $q$ forces that $F$ is an antichain for $\Phi^G$ and all its elements are $\leq_{\Phi^G}$-isolated, then there is no $x \geq n$ with $\Delta^{\pi^q \oplus (A^{M,\Phi} \cup F)}(x) \downarrow = 1$. Define $N$ analogously to the way we did in Case 1.
	\end{highlight}
	
	\begin{highlight}{Case 3:}
		otherwise. We claim that this case cannot obtain, and to show this, break into the following subcases.
	\end{highlight}
	
	\begin{highlight}{Subcase a:}
		$M$ forces that $\Gamma^{G \oplus C^{G,\Phi}}$ and $\Delta^{G \oplus A^{G,\Phi}}$ are both chains for $G$. Since Case 1 does not hold, we can fix a $q_0 \leq_{\mathbb{P}} p^M$, a finite set $F_0 \subseteq I^{M,\Phi}$, and a number $x_0 \geq |\pi^{p^M}|$ such that $q_0$ forces that $F_0$ is a chain for $\Phi^G$ all of whose elements are $\leq_{\Phi^G}$-small, and $\Gamma^{\pi^{q_0} \oplus (C^{M,\Phi} \cup F_0)}(x_0) \downarrow = 1$. Now $\{x \in I^{M,\Phi} : x \geq m_{q_0,F_0}\}$ is computable, so since Case 2 does not hold, we can fix a $q_1 \leq_{\mathbb{P}} q_0$, a finite set $F_1 \subseteq \{x \in I^{M,\Phi} : x \geq m_{q_0,F_0}\}$, and a number $x_1 > x_0$ such that $q_1$ forces that $F_1$ is an antichain for $\Phi^G$ all of whose elements are $\leq_{\Phi^G}$-isolated, and $\Delta^{\pi^{q_1} \oplus (A^{M,\Phi} \cup F_1)}(x_1) \downarrow = 1$. By passing to an extension if necessary, we may assume $|\pi^{q_1}| > \max F_0 \cup F_1$.
		
		Since $q_1 \leq_{\mathbb{P}} p^M$, $\pi^{q_1}$ respects $p^M$. So by part 1 of Lemma \ref{lem:partialext}, we can choose $q \leq_{\mathbb{P}} p^M$ with $\pi^q = \pi^{q_1}$ such that every $z$ with $|\pi^{p^M}| \leq z < |\pi^{q_1}|$ is $q$-isolated. In particular, both $x_0$ and $x_1$ are $q$-isolated. By our use conventions, we have that $\Phi^{\pi^q}$ and $\Phi^{\pi^{q_1}}$ agree below $|\pi^{q_1}|$, so $q$ forces that the $\leq_{\Phi^G}$-largest element of $F_0$ is $\leq_{\Phi^G}$-below every element of $F_1$. Let $r \leq_{\mathbb{P}} q$ decide, for each element of $F_0 \cup F_1$, whether it is $\leq_{\Phi^G}$-small or $\leq_{\Phi^G}$-isolated. Then either $r$ forces that the $\leq_{\Phi^G}$-largest element of $F_0$ is $\leq_{\Phi^G}$-small, in which case all elements of $F_0$ are $\leq_{\Phi^G}$-small, or that the $\leq_{\Phi^G}$-largest element of $F_0$ is $\leq_{\Phi^G}$-isolated, in which case all elements of $F_1$ are $\leq_{\Phi^G}$-isolated.
		
		If $r$ forces that the elements of $F_0$ are all $\leq_{\Phi^G}$-small, define $N$ as follows. Let $p^N = r$, $C^{N,\Phi} = C^{M,\Phi} \cup F_0$, and $A^{N,\Phi} = A^{M,\Phi}$. Choose $m$ larger than $m_{r,F_0}$ and the use of $\Gamma^{\pi^r \oplus (C^{M,\Phi} \cup F_0)}(x_0)$, let $I^{N,\Phi} = \{x \in I^{M,\Phi} : x \geq m\}$, and let the rest of $N$ agree with $M$. Then $N \leq_{\mathbb{M}} M$, and $N$ forces that the set defined by $\Gamma^{G \oplus C^{G,\Phi}}$ contains $x_0$, and that $x_0$ is $\leq_G$-isolated. But this cannot be, since $M$ forces that $\Gamma^{G \oplus C^{G,\Phi}}$ defines a chain for $G$.
		
		If $r$ forces that the elements of $F_1$ are $\leq_{\Phi^G}$-isolated, we proceed similarly. Let $p^N = r$, $A^{N,\Phi} = A^{M,\Phi} \cup F_1$, and $C^{N,\Phi} = C^{M,\Phi}$. Choose $m$ larger than $m_{r,F_1}$ and the use of $\Delta^{\pi^r \oplus (A^{M,\Phi} \cup F_1)}(x_1)$, let $I^{N,\Phi} = \{x \in I^{M,\Phi} : x \geq m\}$, and let the rest of $N$ agree with $M$. Again, $N \leq_{\mathbb{M}} M$, and we have a contradiction because $N$ forces that the set defined by $\Delta^{G \oplus A^{G,\Phi}}$ contains $x_1$, which is $\leq_G$-isolated.
	\end{highlight}
	
	\begin{highlight}{Subcase b:}
		$M$ forces that $\Gamma^{G \oplus C^{G,\Phi}}$ and $\Delta^{G \oplus A^{G,\Phi}}$ are both antichains for $G$. The argument is analogous to the previous subcase.
	\end{highlight}
	
	\begin{highlight}{Subcase c:}
		$M$ forces that $\Gamma^{G \oplus C^{G,\Phi}}$ is a chain for $G$ and $\Delta^{G \oplus A^{G,\Phi}}$ an antichain. The argument is similar to the previous two subcases, but we must take slightly greater care in defining $N$. Fix $q_0 \leq_{\mathbb{P}} p^M$, $F_0 \subseteq I^{M,\Phi}$, and $x_0 \geq |\pi^{p^M}|$ as in Subcase a. Without loss of generality, $|\pi^{q_0}| > x_0$, so $x_0$ must be $q_0$-non-isolated since $M$ forces that $\Gamma^{G \oplus C^{G,\Phi}}$ is a chain for $G$. Say $a^{q_0}(x_0) = (\small,t)$; the case where $x_0$ is $q_0$-large is symmetric. By the failure of Case 2, fix a $q_1 \leq_{\mathbb{P}} q_0$, a finite set $F_1 \subseteq \{x \in I^{M,\Phi} : x \geq m_{q_0,F_0}\}$, and a number $x_1 > \max\{x_0,t\}$ such that $q_1$ forces that $F_1$ is an antichain for $\Phi^G$ all of whose elements are $\leq_{\Phi^G}$-isolated, and $\Delta^{\pi^{q_1} \oplus (A^{M,\Phi} \cup F_1)}(x_1) \downarrow = 1$. In particular, $x_0 <_{\pi^{q_1}} x_1$. We may assume $|\pi^{q_1}| > \max F_0 \cup F_1$.
		
		By part 2 of Lemma \ref{lem:partialext}, choose $q \leq_{\mathbb{P}} p^M$ with $\pi^q = \pi^{q_1}$ such that $x_0$ is $q$-isolated and $x_1$ is $q$-large. Let $r \leq_{\mathbb{P}} q$ decide, for each element of $F_0 \cup F_1$, whether it is $\leq_{\Phi^G}$-small or $\leq_{\Phi^G}$-isolated. Now as in Subcase a, either $r$ forces that all elements of $F_0$ are $\leq_{\Phi^G}$-small, or that all elements of $F_1$ are $\leq_{\Phi^G}$-isolated. In either case, we define $N$ as in Subcase a. If $r$ forces that all elements of $F_0$ are $\leq_{\Phi^G}$-small, we obtain a contradiction because $N$ forces that the set defined by $\Gamma^{G \oplus C^{G,\Phi}}$ contains $x_0$, which $\leq_G$-isolated, even though $M$ forces that $\Gamma^{G \oplus C^{G,\Phi}}$ is a chain for $G$. And if $r$ forces that all elements of $F_1$ are $\leq_{\Phi^G}$-isolated, we obtain a contradiction because $N$ forces that the set defined by $\Delta^{G \oplus A^{G,\Phi}}$ contains $x_1$, which is $\leq_G$-large, even though $M$ forces that $\Delta^{G \oplus A^{G,\Phi}}$ is an antichain for $G$.
		
	\end{highlight}
	
	\begin{highlight}{Subcase d:}
		$M$ forces that $\Gamma^{G \oplus C^{G,\Phi}}$ is an antichain for $G$ and $\Delta^{G \oplus A^{,\Phi}}$ a chain. The argument is analogous to Subcase c, except that when we obtain the $\mathbb{P}$-condition $q_1$ and the numbers $x_0$ and $x_1$, we will have that $x_0$ and $x_1$ are $\leq_{\pi^{q_1}}$-incomparable. Thus, to obtain $q$ as above we will appeal to part 3 of Lemma \ref{lem:partialext} instead of part 2. \qedhere
	\end{highlight}
\end{proof}

We are now ready to prove the theorem.

\begin{highlight}{Theorem~\ref*{thm:wscac}.}
	$\WSCAC \ncred \SCAC$.
\end{highlight}

\begin{proof}
	Let $G = (\omega,\leq_G)$ be the weakly stable partial order given by $\mathcal{F}_{\mathbb{M}}$. Consider any $G$-computable stable partial order, say $\Phi^G$. If this has a $G$-computable infinite chain or antichain, then such a chain or antichain, joined with $G$, can compute no infinite chain or antichain for $G$, by Proposition \ref{prop:partialnocomp}. So suppose $\Phi^G$ has no $G$-computable infinite chain or antichain. By Lemma \ref{lem:Gaddfunctional}, $\Phi \in R^G$, and by Lemma \ref{lem:infcompoents}, each of $C^{G,\Phi}$ and $A^{G,\Phi}$ is infinite, the former a chain for $\Phi^G$ and the latter an antichain. Suppose each of $C^{G,\Phi}$ and $A^{G,\Phi}$, joined with $G$, computes a chain or antichain for $G$, say via functionals $\Gamma$ and $\Delta$, respectively. Then by Lemma \ref{lem:maindiag}, one of $\Gamma^{G \oplus C^{G,\Phi}}$ and $\Delta^{G \oplus A^{G,\Phi}}$ defines a finite set. Thus, one of $C^{G,\Phi}$ or $A^{G,\Phi}$, even joined with $G$, cannot compute any infinite chain or antichain for $G$.
\end{proof}

\bibliography{Papers.bib}
\bibliographystyle{plain}

\end{document}